\documentclass[a4paper,10pt]{article}
\usepackage{amsfonts}
\usepackage{amssymb}
\usepackage{amsmath}
\usepackage{amsthm}
\usepackage{verbatim}
\usepackage{algorithmic}
\theoremstyle{plain}
\newtheorem{theorem}{Theorem}[section]
\newtheorem{lemma}[theorem]{Lemma}

\newtheorem{proposition}[theorem]{Proposition}
\theoremstyle{definition}

\begin{document}
%

  % start of an individual contribution

\title{Team Semantics and Recursive Enumerability}


\author{Antti Kuusisto%
%\thanks{Please note that the LNCS Editorial assumes that all authors have used
%the western naming convention, with given names preceding surnames. This determines
%the structure of the names in the running heads and the author index.}
}

\newcommand{\df}{{D}}
\newcommand{\ifl}{{IF}}
\newcommand{\dtwo}{{\df^2}}
\newcommand{\iftwo}{{\ifl^2}}
\newcommand{\eso}{{ESO}}
\newcommand{\fo}{{FO}}
\newcommand{\fotwo}{{\fo^2}}
\newcommand{\foc}{\logic{FOC}}
\newcommand{\foctwo}{\logic{FOC^2}}
\newcommand{\mA}{{\mathfrak{A}}}
\newcommand{\mB}{{\mathfrak{B}}}
\newcommand{\mG}{{\mathfrak{G}}}
\newcommand{\mH}{{\mathfrak{H}}}
\newcommand{\mD}{{\mathfrak{D}}}
\newcommand{\mM}{{\mathfrak{M}}}
\newcommand{\sat}[1][]{\commandOperator{\problem{Sat}}{#1}}
\newcommand{\finsat}[1][]{\commandOperator{\problem{FinSat}}{#1}}
\newcommand{\calL}{\literal{\mathcal{L}}}
\newcommand{\tiling}[1][]{\commandOperator{\problem{Tiling}}{#1}}
\newcommand{\torus}{\literal{\mathrm{Torus}}}
\newcommand{\I}{\literal{\mathrm{I}}}
%


\date{}

\author{Antti Kuusisto\\
{\small University of Wroc\l aw and Technical University of Denmark}}
\maketitle

\begin{abstract}
\noindent
It is well known that dependence logic captures the complexity class NP, and it has recently been shown that inclusion logic captures P on ordered models. These results demonstrate that team semantics offers interesting new possibilities for descriptive complexity theory. In order to properly understand the connection between team semantics and descriptive complexity, we introduce an extension $\mathrm{D}^*$ of dependence logic that can define exactly all recursively enumerable classes of finite models.
Thus $\mathrm{D}^*$ provides an approach to computation alternative to Turing machines.
The essential novel feature in $\mathrm{D}^*$ is an operator that can extend the domain of the
considered model by a finite number of fresh elements.
%
%Thus exactly corresponds to Turing machines.
%Thus exactly corresponds to Turing machines.
Due to the close relationship between generalized quantifiers and oracles,
we also investigate generalized quantifiers in team semantics.
We show that monotone  
quantifiers of type (1) can be canonically eliminated from quantifier extensions of
first-order logic by introducing corresponding generalized dependence atoms. 
\end{abstract}

\section{Introduction}

In this article we study logics based on \emph{team semantics}.
Team semantics was originally conceived by Hodges
\cite{ho97} in the context of  IF-\emph{logic} \cite{hisa89}.
On the intuitive level, team semantics provides an alternative compositional approach
to systems based on game-theoretic semantics. The compositional approach
simplifies the more traditional game-theoretic approaches in several ways.
In \cite{va07}, V\"{a}\"{a}n\"{a}nen introduced \emph{dependence logic} ($\mathrm{D}$),
which is a novel approach to IF-logic based on new atomic formulae $=\hspace{-1.1mm}(x_1,...,x_k,y)$
that can be interpreted to mean that the choice for the value of $y$ is
\emph{functionally determined} by
the choices for the values of $x_1,...,x_k$ in a semantic game.
After the introduction of dependence logic, research on logics based on team semantics
has been \emph{very active.}
Several different logics with different applications have been investigated.
Currently the two most important systems studied in the field in addition to 
dependence logic are \emph{independence logic} \cite{independencelogic} of 
Gr\"{a}del and V\"{a}\"{a}n\"{a}nen and
\emph{inclusion logic} \cite{inclusionlogic} of Galliani.
Independence logic is a
variant of dependence logic that extends first-order logic by
new atomic formulae $x_1,...,x_k\, \bot\, y_1,...,y_k$
with the intuitive meaning that the interpretations of
the variables $x_1,...,x_k$ are \emph{independent} of
the interpretations of the variables $y_1,...,y_k$.
Inclusion logic extends first-order logic by
atomic formulae $x_1,...,x_k\, \subseteq\, y_1,...,y_k$,
whose intuitive meaning is that each tuple interpreting
the variables $x_1,...,x_k$ must also be a
tuple that interprets $y_1,...,y_k$.
Exclusion logic, also introduced in \cite{inclusionlogic} by Galliani,
is a natural counterpart of inclusion logic with atoms 
$x_1,...,x_k\, |\, y_1,...,y_k$ which state that
the set of tuples interpreting $x_1,...,x_k$ must not overlap
with the set of tuples interpreting $y_1,...,y_k$.
% 

%\cite{hella}

%
It was observed in \cite{va07} and
\cite{independencelogic} that dependence logic and
independence logic are both equi-expressive with \emph{existential
second-order logic},
and thereby capture $\mathrm{NP}$.
Curiously, it was established in \cite{GH2013}
that inclusion logic is equi-expressive with
\emph{greatest fixed point logic} 
and thereby captures $\mathrm{P}$ on finite ordered models.
These results show that team semantics offers a
novel interesting perspective on descriptive complexity theory.
Especially the very close connection between team semantics
and game-theoretic concepts is interesting in this context.
In order \emph{properly understand} the
perspective on descriptive complexity provided
by team semantics, it makes
sense to accomodate the related logics in a
unified umbrella framework that 
exactly characterizes the computational
capacity of Turing machines.
It turns out that there exists a particularly simple
extension of dependence logic that does the job.
Let $\mathrm{D}^*$ denote the logic 
obtained by extending first-order logic
by the atoms of dependence, independence, inclusion, and exclusion
logic, and furthermore, an operator $\mathrm{Ix}$
that extends the domain of the model considered
by a finite number of \emph{fresh} elements.
We show below that $\mathrm{D}^*$ can define exactly
all recursively enumerable classes of finite models.
Since $\mathrm{D}^*$ captures RE, it is not only a
logic but also a model of computation. The
striking simplicity of $\mathrm{D}^*$
and the link between team semantics and game-theory
make $\mathrm{D}^*$ a particularly
interesting system.
There of course exist other logical frameworks
where RE can be easily captured, such as
%
%BGS logic \cite{bgs},
%
%\begin{footnotesize} WHILE
%
%\end{footnotesize}
%
%languages \cite{abiteb}
%
abstract state machines
\cite{guuure1}, \cite{borger}
and the recursive games of
\cite{kuusitur}.
However, $\mathrm{D}^*$
provides a
\emph{simple unified perspective on recent advances in descriptive
complexity based on team semantics.}
The framework of \cite{kuusitur}
resembles $\mathrm{D}^*$
since it provides a perspective on RE
that explains computational notions
via game-theoretic concepts, but 
the approach in \cite{kuusitur}
uses potentially infinite games
(and the article \cite{kuusitur} also lacks a compositional approach). 
The approach provided by
$\mathrm{D}^*$ is different.
%

%
%In addition to capturing $\mathrm{RE}$ by $\mathrm{D}^*$,
%
%we investigate \emph{generalized quantifiers} in team semantics.
%
The notion of a \emph{generalized quantifier} can be seen as a
natural logical counterpart of the computationally
motivated notion of an \emph{oracle}.
Generalized quantifiers have recently been studied from
the point of view of team semantics in
\cite{fredrik}, \cite{freko}, \cite{frekova}, \cite{kuusiquant}.
The focus has been on \emph{monotone} quantifiers.
We establish that type ($1$) monotone generalized quantifiers can be
canonically eliminated from extensions of first-order logic by
introducing corresponding \emph{generalized dependence atoms}.
This result demonstrates that 
%
%corroborates the perspective that
%
team semantics indeed provides a natural
approach to descriptive complexity theory
and computation in general.

\section{Preliminaries}\label{preliminaries}\label{prelies}
We consider only models with a 
\emph{purely relational vocabulary}, i.e., a
vocabulary consisting of relation symbols only.
Therefore, all vocabularies
are below assumed to be purely relational without further warning. 
We let $\mathfrak{A}$, $\mathfrak{B}$, $\mathfrak{C}$, etc., 
denote models; $A$, $B$ and $C$ denote the domains of the models
$\mathfrak{A}$, $\mathfrak{B}$ and $\mathfrak{C}$, respectively.
%

%
%We assume that the reader is familiar with the basics of computational complexity theory.
%
%In this article we are interested in the complexity of the satisfiability problems of logics.
%
%For a  logic \calL, the \emph{satisfiability problem} $\sat[\calL]$ is the set 
%
%\[\sat[\calL] := \{\varphi \in \calL\mid \text{there is a structure $\mA$ such that $\mA\models \varphi$}\}.\]

%The finite satisfiability problem $\finsat[\calL]$ is the analogue of $\sat[\calL]$ in which we require the structure $\mA$ to be finite. %The following observation will be useful later.

%\begin{remark}\label{eso sat}
%If $\varphi$ is a formula over the vocabulary $\tau$ and
%\[\psi := \exists R_1 \dots \exists R_n \exists f_1 \dots \exists f_m \varphi\]
%with $R_1,\dots,R_n,f_1,\dots,f_m\in \tau$, then $\varphi$ is satisfiable iff the second-order formula $\psi$ is satisfiable.
%\end{remark}
%
%
%
We let $\mathrm{VAR}$
denote a countably infinite
%
%\{\, v_i\ |\ i\in\mathbb{Z}_+\ \}$ be
%
set of  exactly all first-order \emph{variable symbols}.
%
%We shall mainly use metavariables $x,y,z,x_1,x_2$, etc.,
%
%in order to refer to variable symbols in $\mathrm{VAR}$.
%
%We let $\overline{x},\overline{y},\overline{z},\overline{x}_1,\overline{x}_2$, etc.,
%
%denote finite nonempty tuples of variable symbols, i.e., 
%
%tuples in $\mathrm{VAR}^n$ for some $n\in\mathbb{Z}_+$.
%
Let $X\subseteq\mathrm{VAR}$ be a \emph{finite}, possibly empty set.
%
%Let $\mathfrak{A}$ be a model;
%
%we do not allow  models to have an empty domain,
%
%so $A\not=\emptyset$.
%
Let $A$ be a set.
A function $s:X\rightarrow A$ is called an
\emph{assignment} with domain $X$ and codomain $A$.
%
%
%
%Let $s$ be an assignment with domain $X$ and codomain $A$.
%
We let $s[a/x]$ denote the assignment 
with domain $X\cup\{x\}$ and 
codomain $A\cup\{a\}$ defined such that 
$s[a/x](y) = a$ if $y = x$,
and 
$s[a/x](y) = s(a)$ if $y \not= x$.
Let $T$ be a set.
%
%where $\mathcal{P}$ denotes the power set operator.
%
%
%
We define 
$s[\,  T/x\, ]\ =\ \{\ s[a/x]\ |\ a\in T\ \}.$
%
%
%
%Note that $s[\, \overline{x}/\emptyset\, ] = \emptyset$.
%

%
Let $X\subseteq\mathrm{VAR}$ be a finite, possibly empty set.
%
%first-order variable symbols.
%
Let $U$ be a set of assignments
$s:X\rightarrow A$. Such a set $U$ is a \emph{team}
with \emph{domain} $X$ and codomain $A$.
Note that the empty set is a team with codomain $A$, as is the set $\{\emptyset\}$
containing only the empty assignment.
The team $\emptyset$ does not have a unique
domain; any finite subset of $\mathrm{VAR}$ is a
domain of $\emptyset$.
The domain of the team $\{\emptyset\}$ is $\emptyset$.
The domain of team $U$ is denoted by $\mathit{Dom}(U)$.
%
%
%
%Let $V$ be a team with the domain $X$ and codomain $A$.
%
Let $T$ be a set.
We define $U[\, T/x\, ] := \{\ s[a/x]\ |\ a\in T,\ s\in U\ \}$.
%
%
%
%Let $V$ be a team with domain $X$ and codomain $A$.
%
%Let $S\subseteq A$.
%
Let $f:U\rightarrow \mathcal{P}(T)$ be a function,
where $\mathcal{P}$ denotes the power set operator.
%
%where $\mathcal{P}$ denotes the power set operator.
%
We define
$U[\, f/x\, ] :=\bigcup\limits_{s\, \in\, U}\ s[\, f(s)/x\, ].$
Let $V$ be a team. Let $k\in\mathbb{Z}_+$,
where $\mathbb{Z}_+$ denotes the  positive integers.
Let $x_1,...,x_k\in\mathit{Dom}(V)$. Define
$\mathit{Rel}\bigl(V,(x_1,...,x_k)\bigr) := \{\bigl(s(x_1),...,s(x_k)\bigr)\ |\ s\in V\}.$
%
%
%
%where $s$ extends to interpret terms with constant and function symbols in the obvious way.
%

%
We then define \emph{lax team semantics}
for formulae of first-order logic ($\mathrm{FO}$).
%
%\begin{definition}\label{def:semantics}
%
As usual in investigations related to team semantics,
formulae are assumed to be in \emph{negation normal form},
i.e., negations occur only in front of atomic formulae.
Let $\mA$ be a model and $U$ a team with codomain $A$.
Let $\models_{\mathrm{FO}}$ denote the ordinary Tarskian satisfaction relation of first-order logic,
i.e., $\mathfrak{A},s\models_{\mathrm{FO}}\varphi$ means that the model $\mathfrak{A}$
satisfies the first-order formula $\varphi$ under the assignment $s$.
We define 
\[
\begin{array}{lll}
 \mathfrak{A},U\models x=y & \ \Leftrightarrow \ & 
          \forall s\in U\bigl(\mathfrak{A},s\models_{\mathrm{FO}} x=y\bigr),\\
 \mathfrak{A},U\models \neg x=y & \ \Leftrightarrow \ & 
          \forall s\in U\bigl(\mathfrak{A},s\models_{\mathrm{FO}} \neg x=y\bigr),\\
%
%\\ & &\text{and }
%
% \forall s\in V\bigl(\mathfrak{A},s\not\models_{\mathrm{FO}} x=y\bigr).\\
%
\mathfrak{A},U\models R(x_1,...,x_k)& \ \Leftrightarrow \ & 
          \forall s\in U\bigl(\mathfrak{A},s\models_{\mathrm{FO}} R(x_1,...,x_k)\bigr),\\
 \mathfrak{A},U\models \neg R(x_1,...,x_k)& \ \Leftrightarrow \ & 
          \forall s\in U\bigl(\mathfrak{A},s\models_{\mathrm{FO}}\neg R(x_1,...,x_k)\bigr),\\
%
%& &\text{and }\forall s\in V\bigl(\mathfrak{A},s\not\models_{\mathrm{FO}}R(x_1,...,x_k)\bigr)\\
%
   \mathfrak{A},U\models(\varphi\wedge\psi) & \ \Leftrightarrow \ &
           \mathfrak{A},U\models\varphi\text{ and }
          \mathfrak{A},U\models\psi,\\
   \mathfrak{A},U\models(\varphi\vee\psi) & \ \Leftrightarrow \ &
%
%         \text{there exist teams }U_0,U_1\subseteq U\\
%
%          & &\text{ such that }U_1\cup U_2 = U,\text{ and }\\
%
           \mathfrak{A},U_0\models\varphi\text{ and }
          \mathfrak{A},U_1\models\psi\text{ for some }\\
         & &\text{teams }U_0,U_1\subseteq U\text{ such that }
          U_0\cup U_1 = U,\\
   \mathfrak{A},U\models\forall x\, \varphi& \ \Leftrightarrow \ &
           \mathfrak{A},U[\, A/x\, ]\models\varphi,\\
           \mathfrak{A},U\models\exists x\, \varphi
& \ \Leftrightarrow \ & \mathfrak{A}, [\, f/x\, ]\models\varphi\text{ for some }
                         f:U\rightarrow (\mathcal{P}(A)\setminus\emptyset).
%
%& &  f:U\rightarrow (\mathcal{P}(A)\setminus\emptyset).
%
\end{array}
\]
%
%\end{definition}
%
%
%
%The \emph{strict team semantics} is a variant of the above semantics such that
%
%clauses 3 and 4 are replaced by the following clauses.
%
%
%
%\begin{enumerate}
%
%\setcounter{enumi}{2} 
%
%\item $\mA,V\models \psi \vee \varphi$ iff there exist teams $Y$ and $Z$ such
%
%that $X=Y\cup Z$, $Y\cap Z = \emptyset$, $\mA\models_Y \psi$ and
%
%$\mA\models _Z \varphi$.
%
%\item $\mA ,V\models \exists x\, \psi$ iff $\mA \models _{X[F/x]} \psi$ for
%
%some $F\colon X\to \{\, \{a\}\ |\ a\in A\ \}$.
%
%\end{enumerate}
%
%
%
%We let $\mathbb{Z}_+$ denote the set of positive integers.
%
A sentence $\varphi$ is true in $\mA$  ($\mA\models \varphi$)
if $\mA,\{\emptyset\}\models\varphi$. 
%
%Which ever semantics we choose,
%
It is well known and easy to show that
for an $\mathrm{FO}$-formula $\varphi$, we have 
$\mA ,U\models \varphi$ iff\, 
$\mA,s \models_{\mathrm{FO}} \varphi$ for all $s\in U$.
%

% Next we define the concepts of logical consequence and  equivalence for formulae of \df and \ifl. 
% \begin{definition}\label{equiv} Let $\varphi$ and $\psi$ be formulas of \df or \ifl. The formula  $\psi$ is a \emph{logical %consequence} of $\varphi$,
% \[\varphi\rightarrow \psi,  \]
%  if for all models $\mA$ and teams $X$, with $fr(\varphi)\cup fr(\psi)\subseteq
%
%$dom(X)$, and $\mA,V\models\varphi$ we have $\mA,V\models\psi$. The formulas $\varphi$ and $\psi$ are \emph{logically %equivalent},
% \[\varphi\equiv \psi,   \]
%  if $\varphi\rightarrow \psi$ and $\psi\rightarrow \varphi $. 
% \end{definition}
%

%
Let $(i_1,...,i_n)$ be a non-empty sequence of positive integers.
A \emph{generalized quantifier} of the type $(i_1,...,i_n)$ is a class $\mathcal{C}$ of structures
$(A,B_1,...,B_n)$ such that the following conditions hold.
\begin{enumerate}
\item
$A\not=\emptyset$,
and for each $j\in\{1,...,n\}$, we have $B_j \subseteq A^{i_j}$.
\item
If $(A',B_1',...,B_n')\in\mathcal{C}$, and if there is an isomorphism $f:A'\rightarrow A''$
from $(A',B_1',...,B_n')$ to $(A'',B_1'',...,B_n'')$,
then $(A'',B_1'',...,B_n'')\in\mathcal{C}$.
\end{enumerate} 
Let $Q$ be a quantifier of the type $(i_1,...,i_n)$.
Let $A\not=\emptyset$ be a set.
We define $Q^{A}$ to be the set
$\{\ (B_1,...,B_n)\ |\ (A,B_1,...,B_n)\in Q\ \}.$
If $Q$ is a quantifier of type $(1)$,
we define $Q' := \{\ (A,B)\ |\ (A,A\setminus B)\in Q\ \}$.
A quantifier $Q$ of type $(1)$
is said to be \emph{monotone},
if the condition
$ A\subseteq B\ \Rightarrow\ (A\in Q^C\ \Rightarrow\ B\in Q^C)$
holds for all $C\not=\emptyset$ and $A,B\subseteq C$.
Let us next see how the ordinary Tarskian semantic relation $\models_{\mathrm{FO}}$
is extended to deal with languages with generalized quantifiers of type $(1)$.
Let $Q$ be a quantifier of type $(1)$.
Let $\mathrm{FO}(Q)$ denote the extension of
$\mathrm{FO}$ obtained by 
adding a new formula formation rule that constructs $Qx\, \varphi$ from $\varphi$.
Let $\mA$ be a model and $s$ an assignment with codomain $A$.
We define $\mA,s\models_{\mathrm{FO}} Qx\, \varphi$ iff
we have $\{\ a\in A\ |\ \mA,s[ a/x ]\models_{\mathrm{FO}}\varphi\ \}\ \in\ Q^A.$
Note that the formula $\neg Qx\, \varphi$ is equivalent to $Q^d\, x\, \neg\varphi$,
where $(\cdot)^d$ denotes the \emph{dualizing operation}
defined such that $Q^{d} := \{\ (C,D)\ |\ (C,C\setminus D)\not\in Q\ \}$.
Note also that $(Q^d)^d = Q$,
and that if $Q$ is monotone, then so is $Q^d$.
Thus a language of first-order logic extended with
monotone type $(1)$ quantifiers can be represented in negation normal form such that
the resulting language is also essentially an extension
of $\mathrm{FO}$ by monotone type $(1)$ quantifiers.
We next show how to extend lax team semantics to first-order logic
with monotone generalized quantifiers of type $(1)$.
(\emph{We investigate only quantifiers of type $(1)$ for the sake of simplicity and brevity; a
somewhat more general approach will be taken up in the journal version.})
As usual, formulae are taken to be in negation normal form.
We define (cf. \cite{fredrik}) that $\mathfrak{A},U\models Qx\, \varphi$
iff there exists a function $f:U\rightarrow Q^A$ such 
that $\mathfrak{A},U[\, f/x\, ]\models \varphi$.
The following proposition from \cite{fredrik} is straightforward to prove by induction on
the structure of formulae.
\begin{proposition}\label{first-ordercorrespondence}
Let $\varphi$ be a formula of first-order logic extended with generalized
quantifiers. Let $U$ be a team. Then
$\mathfrak{A},U\models\varphi
\text{ iff }\, 
\forall s\in U(\mathfrak{A},s\models_{\mathrm{FO}}\varphi).$ 
\end{proposition}
Let $n\in\mathbb{Z}_+$.
Let $Q$ be a generalized quantifier of type $(i_1,...,i_n)$.
Extend the syntax of first-order logic
with atomic expressions 
$A_{Q}(\overline{x}_1,...,\overline{x}_n),$
where each $\overline{x}_j$ is a tuple of variables of length $i_j$.
(\emph{Negated} generalized \emph{atoms are not allowed,
and we only consider logics in negation normal form.})
Let $U$ be a team whose domain contains all variables
that occur in the tuples $\overline{x}_1,...,\overline{x}_n$.
We extend the lax team semantics defined above such that
$\mathfrak{A},U\models A_{Q}(\overline{x}_1,...,
\overline{x}_n)$ iff
$\bigl(\mathit{Rel}(U,\overline{x}_1),...,
\mathit{Rel}(U,\overline{x}_n)\bigr)\in Q^A.$
The generalized quantifier $Q$ defines a
\emph{generalized atom} $A_{Q}$
of the type
$(i_1,...,i_n)$.
This definition is from \cite{kuusiquant}.
Let $k\in\mathbb{Z}_+$. Let $D_k$ denote the generalized quantifier
that contains exactly all structures $(A,R)$ such that $A\not=\emptyset$
and $R\subseteq A^k$ satisfies the condition that
if $(a_1,...,a_{k-1},b)\in R$ and $(a_1,...,a_{k-1},c)\in R$,
then we have $b = c$. \emph{Dependence logic} ($\mathrm{D}$) is 
the extension of first-order logic in negation normal form with the generalized atoms $A_{D_k}$
for each positive integer $k$. These atoms are
called \emph{dependence atoms}. Below we will write
$=\hspace{-1.15mm}(x_1,...,x_k)$ instead of
$A_{D_k}(x_1,...,x_k)$. We note that dependence logic
is sometimes formulated such that negated atoms $\neg\hspace{-1mm}=\hspace{-1.15mm}(x_1,...,x_k)$
are allowed, but since the semantics then
dictates that $\mathfrak{A},U\models \neg\hspace{-1mm}=\hspace{-1.15mm}(x_1,...,x_k)$
iff $U = \emptyset$, these negated atoms can be replaced by $\exists x(x\not=x)$.
\emph{Inclusion logic} is obtained by extending first-order logic in negation normal form
by atoms $x_1,...,x_k\subseteq y_1,...,y_k$ with the semantics
$\mathfrak{A},U\models\ x_1,...,x_k\subseteq y_1,...,y_k$ iff
$\mathit{Rel}(U,(x_1,...,x_k)) \subseteq \mathit{Rel}(U,(y_1,...,y_k))$.
Here $k$ can be any positive integer.
Similarly, \emph{exclusion logic} extends first-order logic in negation normal form
with atoms $x_1,...,x_k\, |\, y_1,...,y_k$ such that 
$\mathfrak{A},U\models\ x_1,...,x_k\, |\, y_1,...,y_k$ iff
$\mathit{Rel}(U,(x_1,...,x_k))\cap \mathit{Rel}(U,(y_1,..,y_k)) = \emptyset$.
Again $k$ can be any positive integer.
\emph{Independence logic} extends first-order logic in negation normal form
with atoms $x_1,...,x_k\, \bot_{z_1,...,z_m}\, y_1,...,y_n$ such that 
$\mathfrak{A},U\models x_1,...,x_k\, \bot_{z_1,...,z_m}\, y_1,...,y_n$ iff
for all $s,s'\in U$ there exists a $t\in U$ such that
%
%\forall s,s'\in X\, \exists s''\in X: \\
%
%
%
$$\big(\bigwedge_{i\leq m} s(z_i)=s'(z_i)\big) \Rightarrow
\big(\bigwedge_{i\leq k} t(x_i)=s(x_i) \wedge \bigwedge_{i\leq m}
t(z_i)=s(z_i) \wedge \bigwedge_{i\leq n} t(y_i)=s'(y_i) \big).$$
Here $k$, $m$, $n$ can be any positive integers.
Independence logic also contains atoms $x_1,...,x_k\, \bot\, y_1,...,y_n$
such that $\mathfrak{A},U\models x_1,...,x_k\, \bot\, y_1,...,y_n$ iff
for all $s,s'\in U$ there exists a $t\in U$ such that
$\bigwedge_{i\leq k} t(x_i)=s(x_i) 
\text{ and }\bigwedge_{i\leq n} t(y_i)=s'(y_i).$
Here $k$ and $n$ can be any positive integers.
% 

%
%\subsection{$\mathbf{D^*}$ and its extensions}
%

%
Let $\mathfrak{A}$ be a model and $\tau$ its vocabulary.
Let $S\not=\emptyset$ be a set such that $S\cap A = \emptyset$.
We let $\mathfrak{A} + S$ denote the model $\mathfrak{B}$
such that $B = A \cup S$ and $R^{\mathfrak{B}} = R^{\mathfrak{A}}$
for all $R\in\tau$. The model $\mathfrak{B}$ is called a \emph{finite bloating}
of $\mathfrak{A}$.
We then define the logic $\mathrm{D}^*$ that captures recursive enumerability.
In the spirit of team semantics, $\mathrm{D}^*$ is based on 
the use of sets of assignments, i.e., teams, that involve first-order variables.
Let $\mathrm{D}^+$ denote the logic obtained by extending 
first-order logic in negation normal form by all dependence atoms,
independence atoms, inclusion atoms, and exclusion atoms.
$\mathrm{D}^*$ is obtained by extending $\mathrm{D}^+$
by an additional formula formation rule stating that if $\varphi$ is a formula,
then so is $\mathrm{\mathrm{Ix}}\, \varphi$.
We define $\mathfrak{A},U\models \mathrm{Ix}\, \varphi$ iff
there exists a finite bloating $\mathfrak{A}+S$ of $\mathfrak{A}$
such that $\mathfrak{A}+S,U[S/x]\models \varphi$.
We note that since $=\hspace{-1mm}(x_1,...,x_k,y)$ is equivalent to
$y\bot_{x_1,...,x_k}\,y$, dependence
atoms can in fact be eliminated from $\mathrm{D}^*$.
Note that if desired, we can avoid reference to a proper class of possible 
bloatings of $\mathfrak{A}$ by 
letting $A_1 := A\cup\{A\}$ to be the canonical bloating of $A$ by one element
and $A_{k+1} := A_k\cup\{A_k\}$ the bloating of $A$ by $k+1$ elements. 
\section{$\mathbf{D^*}$ Captures RE}\label{sec:satfo2}
%%%%%%%%%%%%%%%%%%%%%%%%%%%%%%%%%%%%%%%%%%%%%%%%%%%%%%%%%%%%%%%%%%%%%%%%%%%%
%
%
%

%
Let $\tau$ be a vocabulary. 
Sentences of \emph{existential second-order logic} ($\mathrm{ESO}$) over $\tau$ are
%s
formulae of the type $\exists X_1...\exists X_k\, \varphi$, where $X_1,...,X_k$ are relation variables
and $\varphi$ a sentence of $\mathrm{FO}$ over $\tau\cup\{X_1,...,X_k\}$.
The symbols $X_1,...,X_k$ are not in $\tau$.
We extend $\mathrm{ESO}$
by defining a logic $\mathcal{L}_{RE}$, whose
$\tau$-sentences are of the type $\mathrm{I}Y\, \psi$, where $Y\not\in\tau$ is a
\emph{unary} relation variable and $\psi$ an $\mathrm{ESO}$-sentence
over $\tau\cup\{Y\}$.
%
%Let $\mathfrak{A}$ be a model over the vocabulary $\tau$.
%
Let $\mathfrak{A}$ be a $\tau$-model.
The semantics of $\mathcal{L}_{RE}$
is defined such that $\mathfrak{A}\models\mathrm{I}Y\, \psi$
iff there exists a finite set $S\not=\emptyset$ such that 
the following conditions hold.
\begin{enumerate}
\item
$A\cap S = \emptyset$.
\item
%Let $\tau$ be the vocabulary of $\mathfrak{A}$.
%
Let $\mathfrak{A}^+$ be the model of the vocabulary $\tau\cup\{Y\}$
with domain $A\cup S$ such that $Y^{\mathfrak{A}^+} = S$ 
and $R^{\mathfrak{A}^+} = R^{\mathfrak{A}}$ for all $R\in\tau$.
We have $\mathfrak{A}^+\models\psi$.
%
%(Note indeed that $\psi$ is an $\mathrm{ESO}$-sentence over $\tau\cup\{Y\}$.)
%
\end{enumerate}
%
%
%
%The logic $\mathcal{L}_{RE}$ is a new logic define
%
%for the purposes of this article.
%

%
As we shall see, the logic $\mathcal{L}_{RE}$ can define in the finite exactly all recursively
enumerable classes of finite models.
Let $\sigma\not=\emptyset$ be a finite set of unary relation symbols and $\mathit{Succ}$ a
binary relation symbol.
A \emph{word model} over 
the vocabulary $\{\mathit{Succ}\}\cup\sigma$
is a model $\mathfrak{A}$ defined as follows.
\begin{enumerate}
\item
The domain $A$ of $\mathfrak{A}$ is a
nonempty finite set.
The predicate $\mathit{Succ}$ is a successor relation over $A$, i.e., a
binary relation corresponding to a linear order, but with maximum out-degree
and in-degree equal to one.
\item
Let $b\in A$ be the smallest element with respect to $\mathit{Succ}$.
We have $b\not\in P^{\mathfrak{A}}$ for all $P\in\sigma$. (This is because 
we do not allow models with the empty domain; the empty word
corresponds to the word model with exactly one element.)
For all $a\in A\setminus\{b\}$, there is exactly one $P\in\sigma$
such that $a\in P^{\mathfrak{A}}$.
\end{enumerate}
Word models canonically encode finite words. For example 
the word $abbaa$ over the alphabet $\{a,b\}$ is encoded by the word model $\mathfrak{M}$
over the vocabulary $\{\mathit{Succ}, P_a,P_b\}$ defined
such that
$M = \{0,...,5\}$ and 
$\mathit{Succ}^{\mathfrak{M}}$ is the canonical successor relation
on $M$,
and we have 
$P_a^{\mathfrak{M}} = \{1,4,5\}$
and 
$P_b^{\mathfrak{M}} = \{2,3\}$.
%
%
%

%
\begin{comment}
%
If $w$ is a finite word, we let $\mathfrak{M}(w)$ denote its encoding
%
by a word model in the way defined above.
%
%If $W$ is a set of finite words,
%
%then $\mathcal{M}(W) = \{\, \mathcal{M}(w)\, |\, w\in W\, \}$.
%
%If $\Sigma\not=\emptyset$ is a finite alphabet, we let $\mathcal{M}(\Sigma)$
%
%denote the vocabulary $\{\, \mathit{Succ}\, \}\cup\{\, P_a\, |\, a\in\Sigma\, \}$.
%

%
We define computation of Turing machines in the usual way that involves a possible
%
\emph{tape alphabet} in addition to an \emph{input alphabet}.
%
The two alphabets are disjoint.
%
Let $\Sigma\not=\emptyset$ be a finite alphabet.
%
Then $\Sigma^*$ is the set of all inputs to a
%
Turing machine $\mathrm{TM}$ with input alphabet $\Sigma$.
%
During computation, $\mathrm{TM}$ may use an additional finite set $S$
%
of tape symbols. The set $S$ is the tape alphabet of $\mathrm{TM}$.
%
\end{comment}
%

%
When investigating computations on structure classes (rather than strings),
Turing machines of course operate on 
\emph{encodings} of structures.
We will use the encoding scheme of \cite{libkin}.
Let $\tau$ be a finite vocabulary and $\mathfrak{A}$ a
finite $\tau$-structure. In order to encode the structure $\mathfrak{A}$ by a binary string,
we first need to define a linear ordering of the domain $A$ of $\mathfrak{A}$.
Let $<^{\mathfrak{A}}$ denote such an ordering.
Let $R\in\tau$ be a $k$-ary relation symbol.
The encoding $\mathit{enc}(R^{\mathfrak{A}})$
of $R^{\mathfrak{A}}$ is the $|A|^k$-bit string defined as follows.
Consider an enumeration of all $k$-tuples over $A$ in
the \emph{lexicographic order} defined with respect to $<^{\mathfrak{A}}$.
In the lexicographic order, $(a_1,...,a_k)$ is smaller than $(a_1',...,a_k')$
iff there exists $i\in\{1,...,k\}$ such that $a_i < a_i'$ 
and $a_j = a_j'$ for all $j < i$.
There are $| A |^k$ tuples in $A^k$, and the string
$\mathit{enc}(R^{\mathfrak{A}})$
is the word
$t \in\{0,1\}^*$
of the length $| A |^k$
such that the bit $t_i$ of $t = t_1\, ...\, t_{| A |^k}$ is $1$
if and only if the $i$-th tuple $(a_1,...,a_k)\in A^{k}$ in the lexicographic order
is in the relation $R^{\mathfrak{A}}$.
The encoding $\mathit{enc}(\mathfrak{A})$ is defined as follows.
We first order the relations in $\tau$. Let $p$ be the number of
relations in $\tau$, and let $R_1,...,R_p$ enumerate the
symbols in $\tau$ according to the order.
We define
$\mathit{enc}(\mathfrak{A})\ :=\ 0^{|A|}\cdot 1\cdot \mathit{enc}(R_1^{\mathfrak{A}})\cdot...\cdot
\mathit{enc}(R_p^{\mathfrak{A}}).$
Notice that the encoding of $\mathfrak{A}$ indeed depends on
the order $<^{\mathfrak{A}}$ and the ordering of the relation symbols in $\tau$,
so $\mathfrak{A}$ in general has several encodings. However, we assume that
$\tau$ is always ordered in some canonical way, so the multiplicity of
encodings results in only due to different orderings of the domain of $\mathfrak{A}$.
Let $\tau$ be a finite vocabulary. A Turing machine $\mathrm{TM}$
defines a \emph{semi-decision algorithm for a class $\mathcal{C}$ of finite $\tau$-models}
iff there is an accepting run for $\mathrm{TM}$ on an input $w\in\{0,1\}^*$ exactly when
$w$ is some encoding of some structure $\mathfrak{A}\in\mathcal{C}$.
\begin{proposition}
In the finite, $\mathcal{L}_{RE}$ can define exactly all recursively enumerable classes of models.
\end{proposition}
\begin{proof}[Sketch]
Let $\mathrm{TM}$ be a Turing machine that defines a semi-decision algorithm for some class of models.
It is routine to write a formula $\varphi_{\mathrm{TM}}\, :=\, \mathrm{I}Y \overline{\exists X}\, \psi$
such that $\mathfrak{A}\models \varphi_{\mathrm{TM}}$ iff
there exists an extension $\mathfrak{B}$ of $\mathfrak{A}$ that consist essentially of a copy of $\mathfrak{A}$
and another part $\mathfrak{C}$ that encodes the computation table of
an accepting computation of $\mathrm{TM}$ on an input $\mathit{enc}({\mathfrak{A}})$.
We can use the predicates in $\overline{\exists X}$ in order to define word models that
encode $\mathit{enc}({\mathfrak{A}})$ and other strings that correspond to the Turing machine
tape at different stages of the computation.
Symbols in $\overline{\exists X}$ can also be used, inter alia, in order to define the
other parts of the computation table and an ordering of the
domain of $\mathfrak{A}$, and also relations that connect $\mathfrak{A}$ to
$\mathfrak{C}$ in order to ensure $\mathfrak{A}$ and $\mathfrak{C}$ are correctly related.
The symbol $Y$ is used in order to see  which points belong to the original model $\mathfrak{A}$.
For the converse, given a sentence $\mathrm{I}Y \overline{\exists X}\, \psi$ 
of $\mathcal{L}_{RE}$, we can define a Turing machine that 
first non-deterministically provides a number $k\in\mathbb{Z}_+$ of fresh points to be added to
the domain of the model considered, and then checks if $\overline{\exists X}\, \psi$
holds in the obtained larger model. 
\end{proof}
If desired, obviously $\mathcal{L}_{RE}$ can be modified without
change in expressivity such that
the set of fresh points labelled by $Y$ can also be possibly empty.
We define $\mathcal{L}_{RE}$ with $Y$ always nonempty simply
because of technical issues related to the treatment of
disjunction in team semantics (see below).
It is also worth noting here that $\mathcal{L}_{RE}$ is a
\emph{rather straightforward characterization of $\mathrm{RE}$}
in terms of a logic that is almost classical. Indeed, $\mathcal{L}_{RE}$
is rather similar to $\mathrm{ESO}$.
Our next aim is to prove 
Lemma \ref{seven}, which essentially provides a
way of encoding a unary relation symbol $Y$
by a corresponding variable symbol $y$ with the help of
inclusion, exclusion, and independence atoms. 
For the purposes of the Lemma, we
first define a translation from dependence
logic to $\mathrm{D}^+$.
Let $\chi$ be a \emph{sentence} of dependence logic over a
vocabulary $\tau$ such that $Y\not\in\tau$. Let $y,v,u,u'$ be variables that
\emph{do not occur} in $\chi$. We next define a translation $T_Y^y(\chi)$
of $\chi$ into $\mathrm{D}^+$ by recursion on the structure of $\chi$. (Strictly speaking,
the variables $v,u,u'$ are fixed parameters of the translation just like $y$ and $Y$,
so we should write $T_Y^{y,v,u,u'}(\chi)$
instead of $T_Y^{y}(\chi)$.
%
%but we will simply
%
%write $T_Y^{y}(\chi)$.
%
The issue here is only
that when a sentence $\chi$ is translated, the auxiliary
variables $y,v,u,u'$ should not occur in $\chi$.)
\begin{enumerate}
\item
$T_Y^{\empty\hspace{0.3mm}y}(R(x_1,...,x_k))\, :=\, R(x_1,...,x_k)$
and
$T_Y^{\empty\hspace{0.3mm}y}(\neg R(x_1,...,x_k))\, :=\, \neg R(x_1,...,x_k)$
\item
$T_Y^{\empty\hspace{0.3mm}y}(x = z)\, :=\, x = z$
and
$T_Y^{\empty\hspace{0.3mm}y}(\neg x = z)\, :=\, \neg x = z$
\item
$T_Y^y(Y(x))\, :=\, x\subseteq y$
and 
$T_Y^y(\neg Y(x))\, :=\, x | y$
\item
$T_Y^y(\, =\hspace{-1mm}(x_1,...,x_k)\, )\, :=\ \ =\hspace{-1mm}(x_1,...,x_k)$
\item
$T_Y^{\empty\hspace{0.3mm}y}(\, (\varphi\wedge\psi)\, )\, :=\, (\, (T_Y^{\empty\hspace{0.3mm}y}(\varphi)\wedge
T_Y^{\empty\hspace{0.3mm}y}(\psi)\, )$
\item\label{joo}
$T_Y^{\empty\hspace{0.3mm}y}(\, (\varphi\vee\psi)\, )\, :=\,
\exists v\bigl(\, v\bot_{\overline{z}}\, y\, \wedge\, \bigl((T_Y^{\empty\hspace{0.3mm}y}(\varphi) \wedge v = u)\vee
(T_Y^{\empty\hspace{0.3mm}y}(\psi) \wedge v = u')\bigr)\, \bigr)$,
where $\overline{z}$ contains exactly all variables quantified superordinate to $(\varphi\vee\psi)$ in $\chi$,
i.e., exactly each $x$ such that $(\varphi\vee\psi)$ is in the scope of $\exists x$ or $\forall x$.
\item\label{joojoo}
Assume $\exists x\, \varphi$ is subordinate to a disjunction in $\chi$, meaning that there is a
subformula $(\alpha \vee\beta)$ of $\chi$ and $\exists x\, \varphi$ is a 
subformula of $(\alpha\vee\beta)$.
We define $T_Y^{\empty\hspace{0.3mm}y}(\exists x\, \varphi)\, :=\, \exists x\,
\bigl(\, x\bot_{\overline{z}}\, yv\, \wedge\, T_Y^{\empty\hspace{0.3mm}y}(\varphi)\, \bigr),$
where $\overline{z}$ contains exactly all variables quantified superordinate to $\exists x\, \varphi$ in $\chi$,
with the exception that $\overline{z}$ never contains $x$; the exception is relevant if $\chi$
contains nested quantification of $x$.
\item\label{joojoojoo}
Assume $\exists x\, \varphi$ is not subordinate to a disjunction in $\chi$.
Then $T_Y^{\empty\hspace{0.3mm}y}(\exists x\, \varphi)\, :=\, \exists x\,
\bigl(\, x\bot_{\overline{z}}\, y\, \wedge\, T_Y^{\empty\hspace{0.3mm}y}(\varphi)\, \bigr),$
where $\overline{z}$ contains exactly all variables quantified superordinate to $\exists x\, \varphi$ in $\chi$,
with the exception that $\overline{z}$ never contains $x$.
\item
$T_Y^{\empty\hspace{0.3mm}y}(\forall x\, \varphi)\, :=\, \forall x\, (T_Y^{\empty\hspace{0.3mm}y}(\varphi))$
\end{enumerate}
%

%
%Note that in the above translation, $\overline{z}$ never contains either of the
%
%variables $u$ or $u'$. It turns out, as we will see, that it makes no difference
%
%for our considerations whether
%
%the variables $u$ and $u'$ are included in $\overline{z}$ or not.
%

%
%In order to prove Lemma \ref{seven} below, we need a few
%
%auxiliary definitions.
%
Let $\mathfrak{A}$ be a model
such that $|A|\geq 2$. Let $S\subseteq A$.
Let $\chi$ be a sentence of dependence logic and $\varphi$ a
subformula of $\chi$.
Let $(U,V)$ a pair of be teams with codomain $A$ such
that the following conditions hold.
\begin{enumerate}
\item
Call $Z := \mathit{Dom}(U)$.
$Z$ contains exactly all variables quantified superordinate to $\varphi$
in $\chi$. $\mathit{Dom}(V)$ is $Z\cup\{y,u,u'\}$ or $Z\cup\{y,v,u,u'\}$;
we have $v\in \mathit{Dom}(V)$ iff $\varphi$ is subordinate to a disjunction in $\chi$.
\item
We have $U = V\upharpoonright Z$,
i.e., $U = \{\, {s\upharpoonright Z}\ |\ s\in V\, \}$, where $Z = \mathit{Dom}(U)$.
%
%\item
%$U\ =\, V\upharpoonright\mathit{Dom}(U)$.
%
\item
There exists a team $X$ such
that $V = X[S/y]$. (Thus $S\not=\emptyset$
if $V \not= \emptyset$.)
\item
For all $s,t\in V$, we have $s(u) = t(u)\not= t(u') = s(u')$.
In other words, every assignment in $V$ gives exactly the same interpretation to $u$ and to $u'$,
and the interpretation of $u$ is different from that of $u'$.
\end{enumerate}
When $(U,V)$ satisfies the above four conditions, we say
that $(U,V)$ is a \emph{suitable pair} for $\mathfrak{A}$,
$S\subseteq A$, $(y,v,u,u')$ and $(\varphi,\chi)$.
Below $\mathfrak{A}$, $S$, and $(y,v,u,u')$ will always 
be clear from the context (and in fact the same everywhere),
so we may simply talk about suitable pairs for $(\varphi,\chi)$.
Let $\mathfrak{B}$ be a model and $T\subseteq B$ a set.
Let $\tau$ be the vocabulary of $\mathfrak{B}$.
Let $P\not\in\tau$ be a unary relation symbol.
We let $(\mathfrak{B},P\mapsto T)$
denote the expansion of $\mathfrak{B}$
to the vocabulary $\tau\cup\{P\}$ such that $P^{\mathfrak{B}} = T$.
Let $s$ be an assignment with domain $X$.
Let $\{x_1,...,x_k\}$ be a finite
set of variables. We let $s_{-\{x_1,...,x_k\}}$
denote the assignment $s\upharpoonright (X\setminus\{x_1,...,x_k\})$.
\begin{lemma}\label{seven}
Let $\chi$ be a sentence of dependence logic not containing the symbols $y,v,u,u'$.
Let $\mathfrak{A}$ be a model
with at least two elements. Let $Y$ be a unary relation symbol that occurs
neither in $\chi$ nor in the vocabulary of\, $\mathfrak{A}$.
Let $S\subseteq A$. Let $(\{\emptyset\},V)$ be a suitable pair of
for $\mathfrak{A}$, $S$, $(y,v,u,u')$ and $(\chi,\chi)$.
Then we have $\bigl(\mathfrak{A}, Y\mapsto S\bigr),\{\emptyset\} \models \chi$
iff\, $\mathfrak{A},V\models T_Y^y(\chi)$.
\end{lemma}
\begin{proof}
We prove by induction 
on the  structure of $\chi$ that for any subformula $\varphi$ of $\chi$,
the equivalence $\bigl(\mathfrak{A}, Y\mapsto S\bigr),U \models \varphi\, \Leftrightarrow\,
\mathfrak{A},V\models T_Y^y(\varphi)$ holds  
for all suitable pairs $(U,V)$ for $\mathfrak{A}$, $S$, $(y,v,u,u')$ and $(\varphi,\chi)$.
The cases for literals $R(x_1,...,x_k)$, $\neg R(x_1,...,x_k)$, $x=z$ and $\neg x=z$
are clear,
%
%because the variables $y,v,u,u'$ do not occur in them.
%
%For the same reason,
%
as is the case for $=\hspace{-1mm}(x_1,...,x_k)$.
The cases for literals $Y(x)$ and $\neg Y(x)$ follow directly by
the definition of the atoms $x\subseteq y$ and $x|y$.
The cases for $(\varphi\wedge\psi)$ and $\forall x\, \varphi$
follow easily by the semantics of $\wedge$ and $\forall$ and
the induction hypothesis.
%
%The arguments concerining disjunction and existential quantifier
%
%are trickier, so we will discuss them in detail.
%

%
Let $(U,V)$ be a suitable pair for $\mathfrak{A}$, $S$, $(y,v,u,u')$,
and $(\varphi\vee\psi,\chi)$.
Assume that $(\mathfrak{A},Y\mapsto S),U\models \varphi\vee\psi$.
Thus $(\mathfrak{A},Y\mapsto S),U_0\models \varphi$
and $(\mathfrak{A},Y\mapsto S),U_1\models \psi$ for some
$U_0$ and $U_1$ such that $U_0\cup U_1 = U$.
We must show that
\begin{equation}\label{something}
\mathfrak{A},V\ \models\
\exists v \bigl(\, v\bot_{\overline{z}}\, y\, \wedge\,
\bigl((T_Y^{\empty\hspace{0.3mm}y}(\varphi) \wedge v = u)\vee
(T_Y^{\empty\hspace{0.3mm}y}(\psi) \wedge v=u'\, )\bigr)\, \bigr).
\end{equation}
Define a function $g:V\rightarrow (\mathcal{P}(A)\setminus\{\emptyset\})$
so that the following conditions hold.
\begin{enumerate}
\item
If $s_{-\{u,u',v,y\}}\in U_0\setminus U_1$, then $g(s) = \{s(u)\}$.
\item
If $s_{-\{u,u',v,y\}}\in U_1\setminus U_0$, then $g(s) = \{s(u')\}$.
\item
If $s_{-\{u,u',v,y\}}\in U_0\cap U_1$, then $g(s) = \{s(u),s(u')\}$.
\end{enumerate}
Call $Z := V[g/v]$. We must show that
\begin{equation}\label{foru}
\mathfrak{A},Z\models v\bot_{\overline{z}}\, y\, \wedge\,
\bigl((T_Y^{\empty\hspace{0.3mm}y}(\varphi) \wedge v = u)\vee
(T_Y^{\empty\hspace{0.3mm}y}(\psi)
\wedge v=u'\, \bigr)\, \bigr).
\end{equation}
By the definition of $g$, it is clear that $\mathfrak{A},Z\models v\bot_{\overline{z}}\, y$.
To deal with the rest of Equation \ref{foru},
we need to find teams $Z_0$ and $Z_1$ such that
$\mathfrak{A},Z_0\models T_Y^y(\varphi)\wedge v = u$
and $\mathfrak{A},Z_1\models T_Y^y(\psi)\wedge v = u'$,
and furthermore, $Z_0\cup Z_1 = Z$.
Define $Z_0 := \{\, s\in Z\, |\, s(v) = s(u)\, \}$ and
$Z_1 := \{\, s\in Z\, |\, s(v) = s(u')\, \}$.
To see that $(U_0,Z_0)$ and $(U_1,Z_1)$ are
suitable pairs for $(\varphi,\chi)$ and $(\psi,\chi)$, respectively,
we consider the pair $(U_0,Z_0)$. (The argument for $(U_1,Z_1)$
is similar.)
It is clear that the domain of $Z_0$ is
$\mathit{Dom}(U_0)\cup\{y,v,u,u'\}$.
%
%or $\mathit{Dom}(Z_0)\cup\{y,u,u'\}$.
%
The fact that $U_0 = Z_0\upharpoonright \mathit{Dom}(U_0)$
follows essentially by the definition of the function $g$.
Also, $g$ was defined such that for all $s\in V$,
the set $g(s)$ is independent of $s(y)$, and furthermore,
$Z_0$ was defined such that for all $s'\in Z$, whether or not $s'\in Z_0$
holds, is independent of $s'(y)$. Thus there exists a team $X$
such that $Z_0 = X[S/y]$.
Finally, it is clear that $s(u)=t(u)\not= s(u') = t(u')$
holds for all $s,t\in Z_0$.
Thus $(U_0,Z_0)$ is a suitable pair.
A similar argument shows that also $(U_1,Z_1)$
is a suitable pair.
%
%for $\mathfrak{A}$, $S$ and $(y,v,u,u')$.
%
We know that $\mathfrak{A},U_0\models \varphi$, and
since $(U_0,Z_0)$ is a suitable pair, we may apply the induction hypothesis
in order to conclude that $\mathfrak{A},Z_0\models T_Y^y(\varphi)$.
Similarly, we conclude that $\mathfrak{A},Z_1\models T_Y^y(\psi)$.
Due to the definition of $Z_0$ and $Z_1$, it is clear that $Z = Z_0\cup Z_1$,
and also that $\mathfrak{A},Z_0\models v = u$ and $\mathfrak{A},Z_1\models v = u'$.
Since we already saw that $\mathfrak{A},Z\models v\bot_{\overline{z}}\, y$,
we conclude that Equation \ref{foru} holds.
%
%as desired.
%

%
For the converse, we assume that Equation \ref{something}
holds and show that $(\mathfrak{A},Y\mapsto S),U\models (\varphi\vee\psi)$.
Here $(U,V)$ is a suitable pair for $\mathfrak{A}$, $S$, $(y,v,u,u')$, and $(\varphi\vee\psi,\chi)$.
By Equation 1, there exists a
function $h:V\rightarrow(\mathcal{P}(A)\setminus\emptyset)$
such that we have
%
%
%
%\begin{equation}\label{forum}
%
$\mathfrak{A},V[h/v]\models v\bot_{\overline{z}}\, y\, \wedge\,
\bigl((T_Y^{\empty\hspace{0.3mm}y}(\varphi) \wedge v = u)\vee
(T_Y^{\empty\hspace{0.3mm}y}(\psi)
\wedge v=u'\, \bigr)\, \bigr).$
%
%\end{equation}
%
%
%
Call $W := V[h/v]$.
Thus there exist teams $W_0$ and $W_1$ such that
\begin{equation}\label{third}
\mathfrak{A}, W_0\models (T_Y^y(\varphi)\wedge v = u)
\text{ and }
\mathfrak{A}, W_1\models (T_Y^y(\psi) \wedge v = u'),
\end{equation}
and furthermore $W = W_0\cup W_1$.
Define $U_0 = W_0\upharpoonright\mathit{Dom}(U)$ and $U_1 = W_1\upharpoonright\mathit{Dom}(U)$.
Since $\mathfrak{A},W\models v\bot_{\overline{z}}\, y$,
there exists a team $X$ such that $W = X[S/y]$.
%
%such that $W = X[S/y]$.
%
Thus, since $\mathfrak{A}, W_0\models v = u$
and $\mathfrak{A}, W_1\models v = u'$,
%
%we observe that
%
there exist teams $X_0$ and $X_1$
such that $W_0 = X_0[S/y]$ and $W_1 = X_1[S/y]$.
Hence $(U_0,W_0)$ and $(U_1,W_1)$
are suitable pairs for $(\varphi,\chi)$ and $(\psi,\chi)$, respectively.
Thus we may apply the induction hypothesis
in order to conclude by Equation \ref{third} that $(\mathfrak{A},Y\mapsto S),U_0\models\varphi$
and $(\mathfrak{A},Y\mapsto S),U_1\models\psi$.
As clearly $U_0\cup U_1 = U$,
we have $(\mathfrak{A}, Y\mapsto S), U\models (\varphi\vee\psi)$.
We then discuss formulae of the type $\exists x\, \varphi$.
We consider the case where $\exists x\, \varphi$ is
subordinate to a disjunction; the case where this does not
hold is similar. Let $(U,V)$ be a suitable pair for $(\exists x\, \varphi,\chi)$.
Assume that $(\mathfrak{A},Y\mapsto S),U\models \exists x\, \varphi$.
Thus there exists a function $f:U\rightarrow (\mathcal{P}(A)\setminus\emptyset)$
such that $(\mathfrak{A},Y\mapsto S),U[f/x]\models \varphi$.
Define a function $g:V\rightarrow (\mathcal{P}(A)\setminus\emptyset)$
such that $g(s) = f(s_{-\{y,u,u',v\}})$ for all $s\in V$.
Clearly we have
$\mathfrak{A},V[g/x]\models x\bot_{\overline{z}}\, yv,$
and clearly $\bigl(U[f/x],V[g/x]\bigr)$ is a 
suitable pair. Therefore, as $(\mathfrak{A},Y\mapsto S),U[f/x]\models \varphi$,
we conclude that $\mathfrak{A},V[g/x]\models T_Y^y(\varphi)$ by
the induction hypothesis.
Thus $\mathfrak{A},V[g/x]\models x\bot_{\overline{z}}\, yv\wedge\, T_Y^y(\varphi)$,
and therefore $\mathfrak{A},V\models
\exists x\bigl(x\bot_{\overline{z}}\, yv\wedge\, T_Y^y(\varphi)\bigr)$.
For the converse, assume that 
$\mathfrak{A},V\models
\exists x\bigl(x\bot_{\overline{z}}\, yv\, \wedge\, T_Y^y(\varphi)\bigr)$.
Thus there exists a function $h:V\rightarrow (\mathcal{P}(A)\setminus\emptyset)$
such that $\mathfrak{A},V[h/x]\models
x\bot_{\overline{z}}\, yv\, \wedge\, T_Y^y(\varphi)$.
Let $F:U\rightarrow V$ be a function
such that for each $s\in U$,
we have $s = {(F(s))}_{-\{y,u,u',v\}}$.
Define the function $k:U\rightarrow (\mathcal{P}(A)\setminus\emptyset)$
such that $k(s) = h(F(s))$ for all $s\in U$.
As $\mathfrak{A},V[h/x]\models x\bot_{\overline{z}}\, yv$,
we see that $\bigl(U[k/x],V[h/x]\bigr)$ is a
suitable pair. Thus we may use the induction hypothesis to
conclude that, since $\mathfrak{A},V[h/x]\models T_Y^y(\varphi)$,
we have $(\mathfrak{A},Y\mapsto S),U[k/x]\models\varphi$.
Thus $(\mathfrak{A},Y\mapsto S),U\models\exists x\, \varphi$.
\end{proof}
Define $\mathbb{T}_Y^{\empty\hspace{0.3mm}y}(\varphi) :=
\exists u\exists u'\bigl(\, u\not = u'\ \wedge\,
=\hspace{-1.3mm}(u)\, \wedge\, =\hspace{-1.3mm}(u')\,
\wedge\, T_Y^y(\varphi)\, ).$
The following Lemma now follows directly.
\begin{lemma}\label{morespacelemma}
Let $\mathfrak{A}$ be a model such that $|A|\geq 2$.
Let $S\subseteq A$ be a nonempty finite set. 
Let $\varphi$ be a sentence of dependence logic. Let $y$
be a variable that does not occur in $\varphi$. Let $Y$ be a
unary symbol that occurs neither in $\varphi$ nor in the
vocabulary of $\mathfrak{A}$. 
Then $(\mathfrak{A}, Y\mapsto S),\{\emptyset\}\models\varphi$ iff\, $\mathfrak{A},\{\emptyset\}[S/y]\models
\mathbb{T}_Y^{\empty\hspace{0.3mm}y}(\varphi)$. 
\end{lemma}
\begin{theorem}
$\mathcal{L}_{RE}$ is contained in $\mathrm{D}^*$.
\end{theorem}
\begin{proof}
It is well known that every sentence $\alpha$ of $\mathrm{ESO}$
translates to an equivalent sentence $\alpha^{\#}$ of dependence logic, see \cite{va07}.
%
%The translation $\alpha^{\#}$ of $\alpha$ 
%
%can be obtained effectively from $\alpha$.
%
We shall use this translation below.
Let $\varphi\, :=\, \mathrm{I} Y\, \overline{\exists X}\psi$
be a sentence of $\mathcal{L}_{RE}$, where $\psi$
is a first-order sentence.
%
%Let $\tau$ be the vocabulary of $\varphi$.
%
%
%
%We have the following chain of equivalences, where
%
%the penultimate equivalence follows by Lemma \ref{morespacelemma}.
%
The following chain of equivalences, where
the penultimate equivalence follows by Lemma \ref{morespacelemma},
settles the current theorem.
\begin{align*}
\mathfrak{A}\models\varphi & \Leftrightarrow \bigl(\mathfrak{A} + S,Y\mapsto S\bigr) \models \overline{\exists X}\psi
\text{ for some finite }S\not=\emptyset\text{ s.t. }S\cap A = \emptyset\\ 
&\Leftrightarrow
\bigl(\mathfrak{A} + S,Y\mapsto S\bigr),\{\emptyset\} \models \bigl(\overline{\exists X}\psi\bigr)^{\#}
\text{ for some finite }S\not=\emptyset\text{ s.t. }S\cap A = \emptyset\\
&\Leftrightarrow
\mathfrak{A} + S,\{\emptyset\}[S/y] \models
\mathbb{T}_Y^{y}\bigl(\bigl(\overline{\exists X}\psi\bigr)^{\#}\bigr)
\text{ for some finite }S\not=\emptyset\text{ s.t. }S\cap A = \emptyset\\
&\Leftrightarrow
\mathfrak{A},\{\emptyset\}\ \models\
\mathrm{I}y\, \mathbb{T}_Y^{\empty\hspace{0.3mm}y}\bigl(\bigl(\overline{\exists X}\psi\bigr)^{\#}\bigr)\
\end{align*}
%
% 
%
%Thus $\mathcal{L}_{RE}$ is contained in $\mathrm{D}^*$.
%
\end{proof}
\begin{theorem}
$\mathrm{D}^*$ is contained in $\mathcal{L}_{RE}$.
\end{theorem}
\begin{proof}
Let $\varphi$ be a sentence of $\mathrm{D}^*$.
Assume $\varphi$ contains
$k$ occurrences of the operator $\mathrm{I}$.
Let $\mathrm{TM}$ be a Turing machine such 
that when given an input model $\mathfrak{A}$, $\mathrm{TM}$
first nondeterministically constructs a tuple
$\overline{n}\in({\mathbb{Z}_+})^k$ that gives for each
occurrence of $\mathrm{I}$ in $\varphi$ a number of
new points to be added to the model. Then $\mathrm{TM}$
checks whether $\mathfrak{A}$ satisfies $\varphi$ with the
given tuple $\overline{n}$ of cardinalites to be added during
the evaluation. $\mathrm{TM}$ is a semi-decision algorithm
corresponding to $\varphi$.
%
%for the model checking of $\mathrm{D}^*$. 
%
\end{proof}

\vspace{-5mm}

\section{Eliminating Monotone Generalized Quantifiers}
%
%$\mathrm{FO}(\mathcal{Q})$
%

%
In this section we show that monotone
quantifiers of type $(1)$
can be eliminated from
quantifier extensions of $\mathrm{FO}$ by
introducing generalized
atoms that are canonically
similar to the quantifiers.
Let $Q$ be a monotone type $(1)$ quantifier
and $\overline{y}$ a tuple of variables of
length $k\in\mathbb{N}$.
We let $A_{Q}(\overline{y},x)$
denote the type $(k+1)$ atom defined such that
$\mathfrak{A},U\models A_{Q}(\overline{y},x)$
iff $\mathit{Rel}(U_0,x)\in Q^A$ for each maximal nonempty team $U_0\subseteq U$
with the property $\forall s,t\in U_0\bigl(s(\overline{y}) = t(\overline{y})\bigr)$.
%
%we have $\mathit{Rel}(U_0,x)\in Q^A$.
%
We call this atom the \emph{$k$-atom induced by $Q$}.
Let $\varphi$ be a formula of $\mathrm{FO}$, possibly extended with
monotone quantifiers of type $(1)$.
%
%By $\varphi^d$ we refer to the formula
%
%obtained by putting the formula $\neg\varphi$ into negation normal form.
%
Recall from Section \ref{prelies}
that $\neg Qx\, \varphi$ is equivalent to $Q^d\, x\, \neg\varphi$,
and if $Q$ is a  is a monotone type $(1)$ quantifier,
then so is $Q^d$.
Let $Q$ be a monotone type $(1)$ quantifier.
We let $\mathrm{FO}(Q)$ denote first-order 
logic in negation normal form and extended by
the quantifiers $Q$ and $Q^d$.
We let $\mathrm{FO}(\mathcal{A}_Q)$ denote
first-order logic in negation normal form and extended
by all $k$-atoms induced by $Q$ and $Q'$ for all $k\in\mathbb{N}$.
We define a translation of formulae of $\mathrm{FO}(Q)$
into $\mathrm{FO}(\mathcal{A}_Q)$ as follows.
Let $\alpha$ be a formula of $\mathrm{FO}(Q)$.
%
%The translation $\alpha^*$ of $\varphi$ is
%
%obtained in the following fashion.
%
First, if necessary, rename bound variables of $\alpha$
so that no free variable is also a bound variable.
Then remove all nested quantification of the same variable
by renaming variables; this means that in the resulting
formula, no quantifier quantifying $x$ is not allowed to
be in the scope of another quantifier quantifying the same
variable $x$. We call the resulting formula \emph{clean}.
We then translate the \emph{clean} formula 
$\beta$ to a formula $\beta^*$, where
$(\cdot)^*$ is defined by the rules below.
Note that $\exists$ is a type $(1)$
monotone generalized quantifier and $\exists^d = \forall$,
so the case for $Q$ below covers also $\exists$ and $\forall$.
\begin{enumerate}
\item
For atoms and negated atoms, $\varphi^* :=\varphi$.
\item
$(\, (\varphi\wedge\psi)\, )^*\, :=\, (\ \varphi^* \wedge
\psi^*\, )$\ and\ $(\, (\varphi\vee\psi)\, )^*\, :=\, (\ \varphi^*\vee \psi^*\, )$.
\item
$(\, Qx\, \varphi\, )^*\ :=\ \forall x\bigl(\, (A_{Q}(\overline{y},x)
\wedge \varphi^*)\vee A_{Q'}(\overline{y},x)\, \bigr),$
where $\overline{y}$ contains exactly all variables quantified superordinate
to $Qx\, \varphi$ in the original clean formula $\beta$, and also
the free variables of $\beta$.
\end{enumerate}
Note that, interestingly, \emph{we eliminated the use of\, $\exists$ altogether}.
\begin{theorem}
Let $Q$ be a generalized quantifier of type $(1)$.
Every formula of $\chi$ of\, $\mathrm{FO}(Q)$ can be effectively
translated into a formula $\chi^*$ of\, $\mathrm{FO}(\mathcal{A}_Q)$
such that $\mathfrak{A},U\models\chi\ \Leftrightarrow\ \mathfrak{A},U\models\chi^*$
for all $\mathfrak{A}$ and $U$.
\end{theorem}
\begin{proof}
Let $\beta$ be a clean formula of $\mathrm{FO}(Q)$.
We will show by induction on the structure of $\beta$
that 
$\mathfrak{A},U\models\varphi\, \Leftrightarrow\,
\mathfrak{A},U\models\varphi^*$\ \
%
%\text{ and }\ \
%
%\mathfrak{A},U\models\varphi^d\, \Leftrightarrow\,
%
%\mathfrak{A},U\models(\varphi^d)^*$$
%
%
%
for all subformulae $\varphi$ of $\beta$.
%
%Both equivalences will be used in the induction,
%
%as we will see.
%

%
The cases for atomic and negated atomic formulae are clear,
as are the cases for $\wedge$ and $\vee$.
We discuss the quantifier case in detail.
Assume
$\mathfrak{A},U\models Qx\, \varphi.$
We assume that $U\not=\emptyset$;
the case for $U=\emptyset$ is straightforward,
as long as it is kept in mind that $\emptyset$ is a function
with domain $\emptyset$.
Let $s\in U$ be an arbitrary assignment.
By Proposition \ref{first-ordercorrespondence},
we have $\mathfrak{A},s\models_{\mathrm{FO}}Qx\, \varphi$.
Let $S(s)$ be the set of \emph{exactly all} assignments $s[a/x]$ ($a\in A$)
such that $\mathfrak{A},s[a/x]\models_{\mathrm{FO}}\varphi$.
As $\mathfrak{A},s\models_{\mathrm{FO}}Qx\, \varphi$,
we have $A(s)\, :=\, \{\ a\in A\ |\ s[a/x]\in S(s)\ \}\in Q^A$.
Notice that $A\setminus A(s)\in  Q'^{A}$.
Define the function $f:U\rightarrow Q^A$ such
that $f(s) = A(s)$ for each $s\in U$.
By Proposition \ref{first-ordercorrespondence},
we have $\mathfrak{A},U[f/x]\models\varphi$.
Let $g:U\rightarrow\mathcal{P}(A)$ be
the function such that $g(s) = A\setminus f(s)$ for all $s\in U$.
%
%Since we defined the set $S$ above to
%
%contain \emph{exactly all} assignments $s[a/x]$ extending $s$ that
%
%satisfy $\varphi$,
%
%we have $\mathfrak{A},t\not\models_{\mathrm{FO}}\varphi$
%
%for each $t\in U[g/x]$.
%

%
Since $\mathfrak{A},U[f/x]\models\varphi$,
we have $\mathfrak{A},U[f/x]\models\varphi^*$
by the induction hypothesis.
Since $f$ maps from $U$ to $Q^A$, we have
%
%
%
%\begin{align}\label{three-equ}
%
$\mathfrak{A},U[f/x]\models A_{Q}(\overline{y},x).$
Thus
$\mathfrak{A},U[f/x]\models A_{Q}(\overline{y},x)\wedge\varphi^*.$
To show that $\mathfrak{A},U\models (Qx\, \varphi)^*$, it
we must prove that
\begin{align}\label{def}
\mathfrak{A},U[A/x]\models (A_{Q}(\overline{y},x)
\wedge \varphi^*)\vee A_{Q'}(\overline{y},x).
\end{align}
%
%
%
%Notice that the team $U[A/x]$ splits into the teams $U[f/x]$
%
%\end{align}
%
%
%

%
Since $U[f/x]\cup U[g/x] = U[A/x]$, it now suffices to show that 
$\mathfrak{A},U[g/x]\models A_{Q'}(\overline{y},x)$.
%
%
%
%Recall that $\mathfrak{A},t\not\models_{\mathrm{FO}}\varphi$
%
%
%
Recall  that $A\setminus A(s)\in Q'^{A}$ for each $s\in U$.
Thus, by the definition of $g$, we have $\mathfrak{A},U[g/x]\models A_{Q'}(\overline{y},x)$.
Therefore $\mathfrak{A},U[g/x]\models A_{Q'}(\overline{y},x)$.
%
%holds, as desired, and hence Equation \ref{def} holds.
%

%
%
%
%$$\mathfrak{A},U[x/f]\models A_{Q}(\overline{y},x)\wedge \varphi^*.$$
%
%
%
%\ \text{ and }\ 
%
%
%
%\mathfrak{A},U[x/f']\models \neg\varphi^*\wedge {\neg A}_{Q}(\overline{y},x).$$
%
%
%

%
%
%
Assume then that $\mathfrak{A},U\models (Qx\, \varphi)^*$.
Therefore Equation \ref{def} holds.
%
%
%
%\begin{align}\label{deffi}
%
%\mathfrak{A},U[A/x]\models (A_{Q}(\overline{y},x)
%
%\wedge \varphi^*)\vee (\varphi^d)^*.
%
%\end{align}
%
%
%
%By Proposition \ref{first-ordercorrespondence},
%
We assume that $U\not=\emptyset$,
for otherwise we are done.
Choose an arbitrary assignment $s\in U$.
We will show that $\mathfrak{A},\{s\}\models Qx\, \varphi$.
Due to Proposition \ref{first-ordercorrespondence},
and since we choose $s$ arbitrarily,
this will directly imply that $\mathfrak{A},U\models Qx\, \varphi$.
Since Equation \ref{def} holds,
we infer that $\mathfrak{A},V\models A_{Q}(\overline{y},x)
\wedge \varphi^*$
and $\mathfrak{A},X\models A_{Q'}(\overline{y},x)$
for some teams $V$ and $X$ such that $V\cup X =  U[A/x]$.
%
%(Note that the teams $V$ and $V'$ may overlap.)
%

%
Let $V_{-x}$ denote the team with domain $\mathit{Dom}(V)\setminus\{x\}$
such that for all assignments $t$, we have $t\in V_{-x}$ iff there exists
some element $a\in A$ such that $t[a/x]\in V$.
Assume first that $s\in V_{-x}$, where $s$ is the 
assignment in $U$ we appointed above.
As $\mathfrak{A},V\models A_{Q}(\overline{y},x)$,
we see that there is a set $B\in Q^A$
such that $s[B/x]\subseteq V$.
As $\mathfrak{A},V\models\varphi^*$,
we have $\mathfrak{A},V\models\varphi$
by the induction hypothesis.
Therefore, using Proposition \ref{first-ordercorrespondence},
we infer that $\mathfrak{A},s[B/x]\models\varphi$.
Thus $\mathfrak{A},\{s\}\models Qx\, \varphi$, as desired.
Assume then that $s\not\in V_{-x}$. Thus $s\in X_{-x}$,
where $X_{-x}$ is of course the team with domain $\mathit{Dom}(X)\setminus\{x\}$
such that for all assignments $t$, we have $t\in X_{-x}$ iff there is 
some $a\in A$ such that $t[a/x]\in X$.
Let $S$ be the set of all assignments $t\in X$
such that there is some $a\in A$ such that $t = s[a/x]$.
Since $s\not\in V_{-x}$ and $V\cup X = U[A/x]$, we have $S = s[A/x]$.
As $\mathfrak{A},X\models A_{Q'}(\overline{y},x)$,
we therefore see that $A\in {Q'}^{A}$.
Thus $\emptyset\in Q^{A}$. Thus, as $Q$ is monotone,
we have $Q^A = \mathcal{P}(A)$.
Thus trivially $\mathfrak{A},s\models_{\mathrm{FO}} Qx\, \varphi$,
whence $\mathfrak{A},\{s\}\models Qx\, \varphi$
by Proposition \ref{first-ordercorrespondence}.
\end{proof}
\section{Conclusions}
We have shown how the standard logics based on team semantics extend
naturally to the simple system $\mathrm{D}^*$
that captures $\mathrm{RE}$. The system $\mathrm{D}^*$ nicely
expands the scope of team semantics from logic to computation.
%
%It remains to be seen how much $\mathrm{D}^*$ can 
%
%be further simplified without losing Turing completeness.
%
It will be interesting to investigate, for example,
\emph{what kind of decidable fragments $\mathrm{D}^*$ has.}
%
%Indeed,
%
%team semantics can be seen as a \emph{natural, game-theoretically
%
%motivated approach to computation.}
%
%We have briefly
%
%discussed how generali
%

\bibliographystyle{plain}
\bibliography{tampere}

\end{document}